\documentclass[12pt]{article}
\usepackage{amsmath}
\usepackage{amsthm}
\usepackage{titlesec}

\titleformat{\subsection}[runin]{\normalfont\bfseries}{\thesubsection}{1em}{}
\usepackage{amsfonts}
\usepackage{fancyhdr}
\usepackage{amssymb}
\usepackage{amscd}
\usepackage{accents}
\usepackage{tikz}
\usetikzlibrary{decorations.markings}
\usetikzlibrary{chains}

\tikzset{node distance=2em, ch/.style={circle,draw,on chain,inner sep=2pt},chj/.style={ch,join},every path/.style={shorten >=4pt,shorten <=4pt},line width=1pt,baseline=-1ex}

\usepackage{mathtools}
\usepackage[T1]{fontenc}
\usepackage{selinput}
\SelectInputMappings{%
	adieresis={ä},
	eacute={é},
	Lcaron={Ľ},
}

\usepackage[all,cmtip]{xy}
\usepackage{hyperref}
\usepackage{graphicx}
\graphicspath{ {images/} }
\usepackage[nottoc,notlot,notlof]{tocbibind}
\hypersetup{
	colorlinks=true,
	linkcolor=blue,
	filecolor=magenta,      
	urlcolor=cyan,
	citecolor=blue,
}

\numberwithin{equation}{subsection}

\newtheorem{thm}{Theorem}[subsection]

\newtheorem{cor}[thm]{Corollary}
\newtheorem{prop}[thm]{Proposition}
\newtheorem{rem}[thm]{Remark}

\newtheorem{lem}[thm]{Lemma}

\newcommand{\red}{\mathrm{red}}
\newcommand{\Lie}{\mathrm{Lie}}
\newcommand{\Aut}{\mathrm{Aut}}
\newcommand{\Hom}{\mathrm{Hom}}

\newcommand{\Ad}{\mathrm{Ad}}

\usepackage{authblk}
\begin{document}

\title{Generators of the pro-$p$ Iwahori and Galois representations}
\author[1]{Christophe Cornut}
\author[2]{Jishnu Ray} 
\affil[1]{CNRS - Institut de Mathématiques de Jussieu - Paris Rive Gauche\\

4, place Jussieu, 75252 Paris Cedex 05, France\\

christophe.cornut@imj-prg.fr}

\affil[2]{Département de Mathématiques, Université Paris-Sud $11$\\

 $91405$ Orsay Cedex, France\\ 
 
 jishnu.ray@u-psud.fr}
\date{}

\maketitle

\begin{abstract}
For an odd prime $p$, we determine a minimal set of topological generators of the pro-$p$ Iwahori subgroup of a split reductive group ${G}$ over $\mathbb{Z}_p$. In the simple adjoint case and for any sufficiently large regular prime $p$, we also construct Galois extensions of $\mathbb{Q}$ with Galois group between the pro-$p$ and the standard Iwahori subgroups of ${G}$.
\end{abstract}

\section{Introduction}
Let $p$ be an odd prime, let $\mathbf{G}$ be a split reductive group
over $\mathbb{Z}_{p}$, fix a Borel subgroup $\mathbf{B}=\mathbf{U}\rtimes\mathbf{T}$
of $\mathbf{G}$ with unipotent radical $\mathbf{U}\triangleleft\mathbf{B}$
and maximal split torus $\mathbf{T}\subset\mathbf{B}$. The Iwahori subgroup $I$ and pro-$p$-Iwahori
subgroup $I(1)\subset I$ of $\mathbf{G}(\mathbb{Z}_{p})$ are defined
\cite[3.7]{Ti79} by 
\begin{align*}
I & =  \{ g \in \mathbf{G}(\mathbb{Z}_{p}): \ \text{red}(g)\in\mathbf{B}(\mathbb{F}_{p})\} ,\\
I(1) & =  \{ g\in\mathbf{G}(\mathbb{Z}_{p}): \ \text{red}(g)\in\mathbf{U}(\mathbb{F}_{p})\} .
\end{align*}
where `red' is the reduction map red:  $\textbf{G}(\mathbb{Z}_p) \rightarrow \textbf{G}(\mathbb{F}_p)$. The subgroups $I$ and $I(1)$ are both open subgroups of $\textbf{G}(\mathbb{Z}_p)$.
Thus $I=I(1)\rtimes T_{tors}$ and $\mathbf{T}(\mathbb{Z}_{p})=T(1)\times T_{tors}$
where $T(1)$ and $T_{tors}$ are respectively the pro$-p$ and torsion
subgroups of $\mathbf{T}(\mathbb{Z}_{p})$. Following $\cite{Green}$
(who works with $\mathbf{G}=\mathbf{GL}_{n}$), we construct in section
$2$ a minimal set of topological generators for $I(1)$. 

More precisely, let $M=X^{\ast}(\mathbf{T})$ be the group of characters
of $\mathbf{T}$, $R\subset M$ the set of roots of $\mathbf{T}$
in $\mathfrak{g}=\mathrm{Lie}(\mathbf{G})$, $\Delta\subset R$ the
set of simple roots with respect to $\mathbf{B}$, $R=\coprod_{c\in\mathcal{C}}R_{c}$
the decomposition of $R$ into irreducible components, $\Delta_{c}=\Delta\cap R_{c}$
the simple roots in $R_{c}$, $\alpha_{c,max}$ the highest positive
root in $R_{c}$. We let $\mathcal{D}\subset\mathcal{C}$
be the set of irreducible components of type $G_{2}$ and for $d\in\mathcal{D}$,
we denote by $\delta_{d}\in R_{d,+}$ the sum of the two simple roots
in $\Delta_{d}$. We denote by $M^{\vee}=X_{\ast}(\mathbf{T})$
the group of cocharacters of $\mathbf{T}$, by $\mathbb{Z} R^{\vee}$
the subgroup spanned by the coroots $R^{\vee}\subset M^{\vee}$ and
we fix a set of representatives $\mathcal{S}\subset M^{\vee}$ for
an $\mathbb{F}_{p}$-basis of 
\[
\left(M^{\vee}/\mathbb{Z}R^{\vee}\right)\otimes\mathbb{F}_{p}=\oplus_{s\in\mathcal{S}}\mathbb{F}_{p}\cdot s\otimes1.
\]
We show (see theorem $\ref{thm:Main}$):

\textbf{Theorem}. \textit{The following elements form a minimal set of topological
generators of the pro-$p$-Iwahori subgroup} $I(1)$ \textit{of} $G=\mathbf{G}(\mathbb{Q}_{p})$:
\begin{enumerate}
\item \textit{The semi-simple elements} $\left\{ s(1+p):s\in\mathcal{S}\right\} $
\textit{of} $T(1)$,
\item \textit{For each} $c\in\mathcal{C}$, \textit{the unipotent elements} $\{x_{\alpha}(1):\alpha\in\Delta_{c}\}$,
\item \textit{For each} $c\in\mathcal{C}$, \textit{the unipotent element} $x_{-\alpha_{c,max}}(p)$,
\item \textit{(If} $p=3$\textit{) For each} $d\in\mathcal{D}$, \textit{the unipotent element} $x_{\delta_{d}}(1)$.
\end{enumerate}
This result generalizes Greenberg $\cite{Green}$ proposition $5.3$, see also Schneider and Ollivier ($\cite{OS}$, proposition $3.64$, part $i$) for $G=SL_2$.\\

Let $\mathbf{T}^{ad}$ be the image of $\mathbf{T}$ in the adjoint
group $\mathbf{G}^{ad}$ of $\mathbf{G}$. The action of $\mathbf{G}^{ad}$
on $\mathbf{G}$ induces an action of $\mathbf{T}^{ad}(\mathbb{Z}_{p})$
on $I$ and $I(1)$ and the latter equips the Frattini quotient $\tilde{I}(1)$
of $I(1)$ with a structure of $\mathbb{F}_{p}[T_{tors}^{ad}]$-module,
where $T_{tors}^{ad}$ is the torsion subgroup of $\mathbf{T}^{ad}(\mathbb{Z}_{p})$
(cf. section 2.12). Any element $\beta$ in $\mathbb{Z} R=M^{ad}=X^{\ast}(\mathbf{T}^{ad})$
induces a character $\beta:T_{ad}^{tors}\rightarrow\mathbb{F}_{p}^{\times}$
and we denote by $\mathbb{F}_{p}(\beta)$ the corresponding simple
($1$-dimensional) $\mathbb{F}_{p}[T_{tors}^{ad}]$-module. With these
notations, the theorem implies that\\

\textbf{Corollary}.
\textit{The} $\mathbb{F}_{p}[T_{tors}^{ad}]$-\textit{module} $\tilde{I}(1)$ \textit{is isomorphic
to}
\[
\mathbb{F}_{p}^{\sharp \mathcal{S}}\oplus\Big(\oplus_{\alpha\in\Delta}\mathbb{F}_{p}(\alpha)\Big)\oplus\Big(\oplus_{c\in\mathcal{C}}\mathbb{F}_{p}(-\alpha_{c,max})\Big)\color{magenta}\Big(\oplus\Big(\oplus_{d\in\mathcal{D}}\mathbb{F}_{p}(\delta_{c})\Big)\mbox{\, if }p=3\Big).
\]

Here $\sharp \mathcal{S}$ is the cardinality of $\mathcal{S}$. 
Suppose from now on in this introduction that $\mathbf{G}$ is simple
and of adjoint type. Then:\\

\textbf{Corollary } \textit{The} $\mathbb{F}_{p}[T_{tors}]$\textit{-module} $\tilde{I}(1)$ \textit{is multiplicity
free unless} $p=3$ \textit{and} $\mathbf{G}$ \textit{is of type} $A_{1}$, $B_{\ell}$
\textit{or} $C_{\ell}$ ($\ell\geq2$), $F_{4}$ \textit{or} $G_{2}$.\\

Let now $K$ be a Galois extension of $\mathbb{Q}$, $\Sigma_{p}$
the set of primes of $K$ lying above $p$. Let $M$ be the compositum of all finite $p$-extensions of $K$ which are unramified outside $\Sigma_p$, a Galois extension over $\mathbb{Q}$. Set $\Gamma$ = Gal$(M/K)$, $\Omega=$ Gal$(K/\mathbb{Q})$ and $\Pi$ = Gal$(M/\mathbb{Q})$. We say that
$K$ is $p$-rational if $\Gamma$ is a free pro$-p$ group, see
$\cite{Moh}$. The simplest example is $K=\mathbb{Q}$, where $\Gamma=\Pi$ is also abelian and $M$ is  the cyclotomic $\mathbb{Z}_p$-extension of $\mathbb{Q}$.  Other examples of $p$-rational fields are $\mathbb{Q}(\mu_p)$ where $p$ is a regular prime. 

Assume $K$ is a $p$-rational, totally complex, abelian extension of $\mathbb{Q}$ and $(p-1)\cdot \Omega=0$. Then Greenberg in $\cite{Green}$ constructs a continuous homomorphism 
\begin{equation*}
\rho_0: \  \text{Gal}(M/\mathbb{Q}) \rightarrow GL_n(\mathbb{Z}_p)
\end{equation*}
such that $\rho_0(\Gamma)$ is the pro-$p$ Iwahori subgroup of $SL_n(\mathbb{Z}_p)$,
assuming that there exists $n$ distinct characters of $\Omega$, trivial or odd,  whose product is the trivial character. 

In section $3$, we are proving results which show the existence of $p$-adic Lie extensions of $\mathbb{Q}$ where the Galois group corresponds to a certain specific $p$-adic Lie algebra. More precisely, for $p$-rational fields, we construct continuous morphisms with open image $\rho: \Pi \rightarrow I$ such that $\rho(\Gamma)=I(1)$. We show in corollary $\ref{cor:main}$ that \\

\textbf{Corollary }\textit{Suppose that} $K$ \textit{is a} $p$\textit{-rational totally complex, abelian extension of} $\mathbb{Q}$ \textit{and} 
$(p-1)\cdot \Omega=0$. \textit{Assume
also that if} $p=3$, \textit{our split simple adjoint group} $\mathbf{G}$
\textit{is not of type} $A_{1}$, $B_{\ell}$
\textit{or} $C_{\ell}$ ($\ell\geq2$), $F_{4}$ \textit{or} $G_{2}$. \textit{Then there is a morphism} $\rho:\Pi\rightarrow I$
  \textit{such that} $\rho(\Gamma)=I(1)$ \textit{if and only if there is morphism} $\overline{\rho}:\Omega\rightarrow T_{tors}$
  \textit{such that the characters} $\alpha\circ\overline{\rho}:\Omega\rightarrow\mathbb{F}_{p}^{\times}$
  \textit{for} $\alpha\in \{\Delta \cup -\alpha_{max}\}$ \textit{are all distinct and belong to} $\hat{\Omega}^{\mathcal{S}}_{odd}$.\\
  
  Here $\hat{\Omega}^{\mathcal{S}}_{odd}$ is a subset of the characters of $\Omega$ with values in $\mathbb{F}_p^{\times}$ (it is defined after proposition $\ref{prop:after}$). Furthermore assuming $K=\mathbb{Q}(\mu_p)$ we show the existence of such a morphism $\overline{\rho}: \Omega \rightarrow T_{tors}$ provided
  that $p$ is a sufficiently large regular prime (cf. section $3.2$):\\
  
  \textbf{Corollary }\textit{There is a constant} $c$ \textit{depending only upon the
  type of} $\mathbf{G}$ \textit{such that if} $p>c$ \textit{is a regular prime, then
  for} $K=\mathbb{Q}(\mu_{p})$, $M$, $\Pi$ \textit{and} $\Gamma$ \textit{as above,
  there is a continuous morphism} $\rho:\Pi\rightarrow I$ \textit{with} $\rho(\Gamma)=I(1)$. \\
  
  The constant $c$ can be determined from lemmas $\ref{eq:findcharacters}$, $\ref{lem:find}$ and remark $\ref{rem:exceptional}$.
    
   In section $2$, we find a minimal set of topological generators of $I(1)$ and study the structure of $\tilde{I}(1)$ as an $\mathbb{F}_p[T^{ad}_{tors}]$-module. In section $3$, assuming our group  $\mathbf{G}$ to be simple and adjoint, we discuss the notion of $p$-rational fields and construct continuous morphisms $\rho:\Pi \rightarrow I$ with open image. 
    
    We would like to thank Marie-France Vignéras for useful discussions and for giving us the reference $\cite{OS}$. We are also deeply grateful to Ralph Greenberg for numerous conversations on this topic.
\section{Topological Generators of the pro-$p$ Iwahori}
This section is organized as follows. In sections $(2.1-2.3)$ we introduce the notations, then section $2.4$ states our main result concerning the minimal set of topological generators of $I(1)$ (see theorem $\ref{thm:Main}$) with a discussion of the Iwahori factorisation in section $2.5$. Its proof for $\mathbf{G}$ simple and simply connected is given in sections $(2.6-2.10)$, where section $2.10$ deals with the case of a group of type $G_2$. The proof for an arbitrary split reductive group over $\mathbb{Z}_p$ is discussed in sections $(2.11-2.14)$. In particular, section $2.14$ establishes the minimality of our set of topological generators. Finally, in section $2.15$ we study the structure of the Frattini quotient $\tilde{I}(1)$ of $I(1)$ as an $\mathbb{F}_p[T_{tors}^{ad}]$-module and determine the cases when it is multiplicity free.

\subsection{~}

Let $p$ be an odd prime,  $\mathbf{G}$ be a split reductive group over $\mathbb{Z}_{p}$.
Fix a pinning of $\mathbf{G}$ \cite[XXIII 1]{SGA3.3r} 
\[
\left(\mathbf{T},M,R,\Delta,(X_{\alpha})_{\alpha\in\Delta}\right).
\]
Thus $\mathbf{T}$ is a split maximal torus in $\mathbf{G}$, $M=X^{\ast}(\mathbf{T})$
is its group of characters, 
\[
\mathfrak{g}=\mathfrak{g}_{0}\oplus\oplus_{\alpha\in R}\mathfrak{g}_{\alpha}
\]
is the weight decomposition for the adjoint action of $\mathbf{T}$
on $\mathfrak{g}=\Lie(\mathbf{G})$, $\Delta\subset R$ is a basis
of the root system $R\subset M$ and for each $\alpha\in\Delta$,
$X_{\alpha}$ is a $\mathbb{Z}_{p}$-basis of $\mathfrak{g}_{\alpha}$.

\subsection{~}

We denote by $M^{\vee}=X_{\ast}(\mathbf{T})$ the group of cocharacters
of $\mathbf{T}$, by $\alpha^{\vee}$ the coroot associated to $\alpha\in R$
and by $R^{\vee}\in M^{\vee}$ the set of all such coroots. We expand
$(X_{\alpha})_{\alpha\in\Delta}$ to a Chevalley system $(X_{\alpha})_{\alpha\in R}$
of $\mathbf{G}$ \cite[XXIII 6.2]{SGA3.3r}. For $\alpha\in R$, we
denote by $\mathbf{U}_{\alpha}\subset\mathbf{G}$ the corresponding
unipotent group, by $x_{\alpha}:\mathbf{G}_{a,\mathbb{Z}_{p}}\rightarrow\mathbf{U}_{\alpha}$
the isomorphism given by $x_{\alpha}(t)=\exp(tX_{\alpha})$. The height
$h(\alpha)\in\mathbb{Z}$ of $\alpha\in R$ is the sum of the coefficients
of $\alpha$ in the basis $\Delta$ of $R$. Thus $R_{+}=h^{-1}(\mathbb{Z}_{>0})$
is the set of positive roots in $R$, corresponding to a Borel subgroup
$\mathbf{B}=\mathbf{U}\rtimes\mathbf{T}$ of $\mathbf{G}$ with unipotent
radical $\mathbf{U}$. We let $\mathcal{C}$ be the set of irreducible
components of $R$, so that 
\[
R=\coprod_{c\in\mathcal{C}}R_{c},\quad\Delta=\coprod_{c\in\mathcal{C}}\Delta_{c},\quad R_{+}=\coprod_{c\in\mathcal{C}}R_{c,+}
\]
with $R_{c}$ irreducible, $\Delta_{c}=\Delta\cap R_{c}$ is a basis
of $R_{c}$ and $R_{c,+}=R_{+}\cap R_{c}$ is the corresponding set
of positive roots in $R_{c}$. We denote by $\alpha_{c,max}\in R_{c,+}$
the highest root of $R_{c}$. We let $\mathcal{D}\subset\mathcal{C}$
be the set of irreducible components of type $G_{2}$ and for $d\in\mathcal{D}$,
we denote by $\delta_{d}\in R_{d,+}$ the sum of the two simple roots
in $\Delta_{d}$.

\subsection{~}

Since $\mathbf{G}$ is smooth over $\mathbb{Z}_{p}$, the reduction
map 
\[
\red:\mathbf{G}(\mathbb{Z}_{p})\rightarrow\mathbf{G}(\mathbb{F}_{p})
\]
is surjective and its kernel $G(1)$ is a normal pro-$p$-subgroup
of $\mathbf{G}(\mathbb{Z}_{p})$. The Iwahori subgroup $I$ and pro-$p$-Iwahori
subgroup $I(1)\subset I$ of $\mathbf{G}(\mathbb{Z}_{p})$ are defined
\cite[3.7]{Ti79} by 
\begin{eqnarray*}
I & = & \left\{ g\in\mathbf{G}(\mathbb{Z}_{p}):\red(g)\in\mathbf{B}(\mathbb{F}_{p})\right\} ,\\
I(1) & = & \left\{ g\in\mathbf{G}(\mathbb{Z}_{p}):\red(g)\in\mathbf{U}(\mathbb{F}_{p})\right\} .
\end{eqnarray*}
Thus $I(1)$ is a normal pro-$p$-sylow subgroup of $I$ which contains
$\mathbf{U}(\mathbb{Z}_{p})$ and 
\[
I/I(1)\simeq\mathbf{B}(\mathbb{F}_{p})/\mathbf{U}(\mathbb{F}_{p})\simeq\mathbf{T}(\mathbb{F}_{p}).
\]
Since $\mathbf{T}(\mathbb{Z}_{p})\twoheadrightarrow\mathbf{T}(\mathbb{F}_{p})$
is split by the torsion subgroup $T_{tors}\simeq\mathbf{T}(\mathbb{F}_{p})$
of $\mathbf{T}(\mathbb{Z}_{p})$,
\[
\mathbf{T}(\mathbb{Z}_{p})=T(1)\times T_{tors}\qquad\mbox{and}\qquad I=I(1)\rtimes T_{tors}
\]
where 
\[
T(1)=\mathbf{T}(\mathbb{Z}_{p})\cap I(1)=\ker\left(\mathbf{T}(\mathbb{Z}_{p})\rightarrow\mathbf{T}(\mathbb{F}_{p})\right)
\]
is the pro-$p$-sylow subgroup of $\mathbf{T}(\mathbb{Z}_{p})$.
Note that 
\begin{eqnarray*}
T(1) & = & \Hom\left(M,1+p\mathbb{Z}_{p}\right)=M^{\vee}\otimes(1+p\mathbb{Z}_{p}),\\
T_{tors} & = & \Hom\left(M,\mu_{p-1}\right)=M^{\vee}\otimes\mathbb{F}_{p}^{\times}.
\end{eqnarray*}

\subsection{~}

Let $\mathcal{S}\subset M^{\vee}$ be a set of representatives for
an $\mathbb{F}_{p}$-basis of 
\[
\left(M^{\vee}/\mathbb{Z}R^{\vee}\right)\otimes\mathbb{F}_{p}=\oplus_{s\in\mathcal{S}}\mathbb{F}_{p}\cdot s\otimes1.
\]

\begin{thm}
\label{thm:Main}The following elements form a minimal set of topological
generators of the pro-$p$-Iwahori subgroup $I(1)$ of $G=\mathbf{G}(\mathbb{Q}_{p})$:
\begin{enumerate}
\item The semi-simple elements $\left\{ s(1+p):s\in\mathcal{S}\right\} $
of $T(1)$.
\item For each $c\in\mathcal{C}$, the unipotent elements $\{x_{\alpha}(1):\alpha\in\Delta_{c}\}$.
\item For each $c\in\mathcal{C}$, the unipotent element $x_{-\alpha_{c,max}}(p)$.
\item (If $p=3$) For each $d\in\mathcal{D}$, the unipotent element $x_{\delta_{d}}(1)$.
\end{enumerate}
\end{thm}

\subsection{~}\label{sub:IwaDec}

By~\cite[XXII 5.9.5]{SGA3.3r} and its proof, there is a canonical
filtration 
\[
\mathbf{U}=\mathbf{U}_{1}\supset\mathbf{U}_{2}\supset\cdots\supset\mathbf{U}_{h}\supset\mathbf{U}_{h+1}=1
\]
of $\mathbf{U}$ by normal subgroups such that for $1\leq i\leq h$,
the product map (in any order)
\[
\prod_{h(\alpha)=i}\mathbf{U}_{\alpha}\rightarrow\mathbf{U}
\]
factors through $\mathbf{U}_{i}$ and yields an isomorphism of group
schemes
\[
\prod_{h(\alpha)=i}\mathbf{U}_{\alpha}\stackrel{\simeq}{\longrightarrow}\overline{\mathbf{U}}_{i},\quad\overline{\mathbf{U}}_{i}=\mathbf{U}_{i}/\mathbf{U}_{i+1}.
\]
By~\cite[XXII 5.9.6]{SGA3.3r} and its proof, 
\[
\overline{\mathbf{U}}_{i}(R)=\mathbf{U}_{i}(R)/\mathbf{U}_{i+1}(R)
\]
for every $\mathbb{Z}_{p}$-algebra $R$. It follows that the product
map 
\[
\prod_{h(\alpha)=i}\mathbf{U}_{\alpha}\times\mathbf{U}_{i+1}\rightarrow\mathbf{U}_{i}
\]
is an isomorphism of $\mathbb{Z}_p$-schemes and by induction, the product map 
\[
\prod_{h(\alpha)=1}\mathbf{U}_{\alpha}\times\prod_{h(\alpha)=2}\mathbf{U}_{\alpha}\times\cdots\times\prod_{h(\alpha)=h}\mathbf{U}_{\alpha}\rightarrow\mathbf{U}
\]
is an isomorphism of $\mathbb{Z}_p$-schemes. Similarly, the product map
\[
\prod_{h(\alpha)=-h}\mathbf{U}_{\alpha}\times\prod_{h(\alpha)=-h+1}\mathbf{U}_{\alpha}\times\cdots\times\prod_{h(\alpha)=-1}\mathbf{U}_{\alpha}\rightarrow\mathbf{U}^{-}
\]
is an isomorphism of $\mathbb{Z}_p$-schemes, where $\mathbf{U}^{-}$ is the unipotent
radical of the Borel subgroup $\mathbf{B}^{-}=\mathbf{U}^{-}\rtimes\mathbf{T}$
opposed to $\mathbf{B}$ with respect to $\mathbf{T}$. Then by \cite[XXII 4.1.2]{SGA3.3r},
there is an open subscheme $\Omega$ of $\mathbf{G}$ (the ``big
cell'') such that the product map 
\[
\mathbf{U}^{-}\times\mathbf{T}\times\mathbf{U}\rightarrow\mathbf{G}
\]
is an open immersion with image $\Omega$. Plainly, $\mathbf{B}=\mathbf{U}\rtimes\mathbf{T}$
is a closed subscheme of $\Omega$. Thus by definition of $I$, $I\subset\Omega(\mathbb{Z}_{p})$
and therefore any element of $I$ (resp. $I(1)$) can be written uniquely
as a product 
\[
\prod_{h(\alpha)=-h}x_{\alpha}(a_{\alpha})\times\cdots\times\prod_{h(\alpha)=-1}x_{\alpha}(a_{\alpha})\times t\times\prod_{h(\alpha)=1}x_{\alpha}(a_{\alpha})\times\cdots\times\prod_{h(\alpha)=h}x_{\alpha}(a_{\alpha})
\]
where $a_{\alpha}\in\mathbb{Z}_{p}$ for $\alpha\in R_{+}$, $a_{\alpha}\in p\mathbb{Z}_{p}$
for $\alpha\in R_{-}=-R_{+}$ and $t\in\mathbf{T}(\mathbb{Z}_{p})$
(resp.~$T(1)$). This is the Iwahori decomposition of $I$ (resp.
$I(1)$). If $I^{+}$ is the group spanned by $\left\{ x_{\alpha}(\mathbb{Z}_{p}):\alpha\in R_{+}\right\} $
and $I^{-}$ is the group spanned by $\left\{ x_{\alpha}(p\mathbb{Z}_{p}):\alpha\in R_{-}\right\} $,
then $I^{+}=\mathbf{U}(\mathbb{Z}_{p})$, $I^{-}\subset\mathbf{U}^{-}(\mathbb{Z}_{p})$
and every $x\in I$ (resp.~$I(1)$) has a unique decomposition $x=u^{-}tu^{+}$
with $u^{\pm}\in I^{\pm}$ and $t\in\mathbf{T}(\mathbb{Z}_{p})$ (resp.~$t\in T(1)$).

\subsection{~}

Suppose first that $\mathbf{G}$ is semi-simple and simply connected.
Then $M^{\vee}=\mathbb{Z}R^{\vee}$, thus $\mathcal{S}=\emptyset$.
Moreover, everything splits according to the decomposition $R=\coprod R_{c}$:
\[
\mathbf{G}=\prod\mathbf{G}_{c},\quad\mathbf{T}=\prod\mathbf{T}_{c},\quad\mathbf{B}=\prod\mathbf{B}_{c},\quad I=\prod I_{c}\quad\mbox{and}\quad I(1)=\prod I_{c}(1).
\]
To establish the theorem in this case, we may thus furthermore assume
that $\mathbf{G}$ is simple. From now on until section~\ref{sub:Isogenies},
we therefore assume that 
\[
\mathbf{G}\mbox{\, is (split) simple and simply connected.}
\]

\subsection{~}

As a first step, we show that 
\begin{lem}
\label{lem:SpanT}The group generated by $I^{+}$ and $I^{-}$contains
$T(1)$. \end{lem}
\begin{proof}
Since $\mathbf{G}$ is simply connected, 
\[
\prod_{\alpha\in\Delta}\alpha^{\vee}:\prod_{\alpha\in\Delta}\mathbf{G}_{m,\mathbb{Z}_{p}}\rightarrow\mathbf{T}
\]
is an isomorphism, thus 
\[
T_{c}(1)=\prod_{\alpha\in\Delta}\alpha^{\vee}(1+p\mathbb{Z}_{p}).
\]
Now for any $\alpha\in\Delta$, there is a unique morphism \cite[XX 5.8]{SGA3.3r}
\[
f_{\alpha}:\mathbf{SL}(2)_{\mathbb{Z}_{p}}\rightarrow\mathbf{G}
\]
such that for every $u,v\in\mathbb{Z}_{p}$ and $x \in \mathbb{Z}_p^{\times}$,
\[
f_{\alpha}\left(\begin{array}{cc}
1 & u\\
0 & 1
\end{array}\right)=x_{\alpha}(u),\quad f_{\alpha}\left(\begin{array}{cc}
1 & 0\\
v & 1
\end{array}\right)=x_{-\alpha}(v)\quad\mbox{and}\quad f_{\alpha}\left(\begin{array}{cc}
x & 0\\
0 & x^{-1}
\end{array}\right)=\alpha^{\vee}(x).
\]
Since for every $x\in 1+p\mathbb{Z}_p$ \cite[XX 2.7]{SGA3.3r},
\[
\left(\begin{array}{cc}
1 & 0\\
x^{-1}-1 & 1
\end{array}\right)\left(\begin{array}{cc}
1 & 1\\
0 & 1
\end{array}\right)\left(\begin{array}{cc}
1 & 0\\
x-1 & 1
\end{array}\right)\left(\begin{array}{cc}
1 & -x^{-1}\\
0 & 1
\end{array}\right)=\left(\begin{array}{cc}
x & 0\\
0 & x^{-1}
\end{array}\right)
\]
in $\mathbf{SL}(2)(\mathbb{Z}_{p})$, it follows that $\alpha^{\vee}(1+p\mathbb{Z}_{p})$
is already contained in the subgroup of $\mathbf{G}(\mathbb{Z}_{p})$
generated by $x_{\alpha}(\mathbb{Z}_{p}^{\times})$ and $x_{-\alpha}(p\mathbb{Z}_{p})$.
This proves the lemma.
\end{proof}

\subsection{~}

Recall from~\cite[XXI 2.3.5]{SGA3.3r} that for any pair of non-proportional
roots $\alpha\neq\pm\beta$ in $R$, the set of integers $k\in\mathbb{Z}$
such that $\beta+k\alpha\in R$ is an interval of length at most $3$,
i.e.~there are integers $r\geq1$ and $s\geq0$ with $r+s\leq4$
such that 
\[
R\cap\{\beta+\mathbb{Z}\alpha\}=\{\beta-(r-1)\alpha,\cdots,\beta+s\alpha\}.
\]
The above set is called the $\alpha$-chain through $\beta$ and any
such set is called a root chain in $R$. Let $\left\Vert -\right\Vert :R\rightarrow\mathbb{R}_{+}$
be the length function on $R$. 
\begin{prop}
\label{prop:ChevalleyFormulas}Suppose $\left\Vert \alpha\right\Vert \leq\left\Vert \beta\right\Vert $.
Then for any $u,v\in\mathbf{G}_{a}$ the commutator 
\[
\left[x_{\beta}(v):x_{\alpha}(u)\right]=x_{\beta}(v)x_{\alpha}(u)x_{\beta}(-v)x_{\alpha}(-u)
\]
is given by the following table, with $(r,s)$ as above: 
\[
\begin{array}{cl}
(r,s) & [x_{\beta}(v):x_{\alpha}(u)]\\
(-,0) & 1\\
(1,1) & x_{\alpha+\beta}(\pm uv)\\
(1,2) & x_{\alpha+\beta}(\pm uv)\cdot x_{2\alpha+\beta}(\pm u^{2}v)\\
(1,3) & x_{\alpha+\beta}(\pm uv)\cdot x_{2\alpha+\beta}(\pm u^{2}v)\cdot x_{3\alpha+\beta}(\pm u^{3}v)\cdot x_{3\alpha+2\beta}(\pm u^{3}v^{2})\\
(2,1) & x_{\alpha+\beta}(\pm2uv)\\
(2,2) & x_{\alpha+\beta}(\pm2uv)\cdot x_{2\alpha+\beta}(\pm3u^{2}v)\cdot x_{\alpha+2\beta}(\pm3uv^{2})\\
(3,1) & x_{\alpha+\beta}(\pm3uv)
\end{array}
\]
The signs are unspecified, but only depend upon $\alpha$ and $\beta$. \end{prop}
\begin{proof}
This is \cite[XXIII 6.4]{SGA3.3r}.\end{proof}
\begin{cor}
\label{cor:SpanAlpha+Beta}If $r+s\leq3$ and $\alpha+\beta\in R$
(i.e.~$s\geq1$), then for any $a,b\in\mathbb{Z}$, the subgroup
of $G$ generated by $x_{\alpha}(p^{a}\mathbb{Z}_{p})$ and $x_{\beta}(p^{b}\mathbb{Z}_{p})$
contains $x_{\alpha+\beta}(p^{a+b}\mathbb{Z}_{p})$. \end{cor}
\begin{proof}
This is obvious if $(r,s)=(1,1)$ or $(2,1)$ (using $p\neq2$ in
the latter case). For the only remaining case where $(r,s)=(1,3)$,
note that
\[
[x_{\beta}(v):x_{\alpha}(u)][x_{\beta}(w^{2}v):x_{\alpha}(uw^{-1})]^{-1}=x_{\alpha+\beta}(\pm uv(1-w)).
\]
Since $p\neq2$, we may find $w\in\mathbb{Z}_{p}^{\times}$ with $(1-w)\in\mathbb{Z}_{p}^{\times}$.
Our claim easily follows. \end{proof}
\begin{lem}
If $R$ contains any root chain of length $3$, then $\mathbf{G}$
is of type $G_{2}$.\end{lem}
\begin{proof}
Suppose that the $\alpha$-chain through $\beta$ has length $3$.
By \cite[XXI 3.5.4]{SGA3.3r}, there is a basis $\Delta'$ of $R$
such that $\alpha\in\Delta'$ and $\beta=a\alpha+b\alpha'$ with $\alpha'\in\Delta'$,
$a,b\in\mathbb{N}$. The root system $R'$ spanned by $\Delta'=\{\alpha,\alpha'\}$
\cite[XXI 3.4.6]{SGA3.3r} then also contains an $\alpha$-chain of
length $3$. By inspection of the root systems of rank $2$, for instance
in~\cite[XXIII 3]{SGA3.3r}, we find that $R'$ is of type $G_{2}$.
In particular, the Dynkin diagram of $R$ contains a triple edge (linking
the vertices corresponding to $\alpha$ and $\alpha'$), which implies
that actually $R=R'$ is of type $G_{2}$.
\end{proof}

\subsection{~}

We now establish our theorem~\ref{thm:Main} for a group $\mathbf{G}$
which is simple and simply connected, but not of type $G_{2}$. 
\begin{lem}
\label{lem:SpanI+}The group $I^{+}$ is generated by $\left\{ x_{\alpha}(\mathbb{Z}_{p}):\alpha\in\Delta\right\} $.\end{lem}
\begin{proof}
Let $H\subset I^{+}$ be the group spanned by $\left\{ x_{\alpha}(\mathbb{Z}_{p}):\alpha\in\Delta\right\} $.
We show by induction on $h(\gamma)\geq1$ that $x_{\gamma}(\mathbb{Z}_{p})\subset H$
for every $\gamma\in R_{+}$. If $h(\gamma)=1$, $\gamma$ already
belongs to $\Delta$ and there is nothing to prove. If $h(\gamma)>1$,
then by \cite[VI.1.6 Proposition 19]{BoLie46}, there is a simple
root $\alpha\in\Delta$ such that $\beta=\gamma-\alpha\in R_{+}$.
Then $h(\beta)=h(\gamma)-1$, thus by induction $x_{\beta}(\mathbb{Z}_{p})\subset H$.
Since also $x_{\alpha}(\mathbb{Z}_{p})\subset H$, $x_{\gamma}(\mathbb{Z}_{p})\subset H$
by Corollary~\ref{cor:SpanAlpha+Beta}.\end{proof}
\begin{lem}
\label{lem:SpanI-}The group generated by $I^{+}$ and $x_{-\alpha_{max}}(p\mathbb{Z}_{p})$
contains $I^{-}$.\end{lem}
\begin{proof}
Let $H\subset I$ be the group spanned by $I^{+}$ and $x_{-\alpha_{max}}(p\mathbb{Z}_{p})$.
We show by descending induction on $h(\gamma)\geq1$ that $x_{-\gamma}(p\mathbb{Z}_{p})\subset H$
for every $\gamma\in R_{+}$. If $h(\gamma)=h(\alpha_{max})$, then
$\gamma=\alpha_{max}$ and there is nothing to prove. If $h(\gamma)<h(\alpha_{max})$,
then by \cite[VI.1.6 Proposition 19]{BoLie46}, there is a pair of
positive roots $\alpha,\beta$ such that $\beta=\gamma+\alpha$. Then
$h(\beta)=h(\gamma)+h(\alpha)>h(\gamma)$, thus by induction $x_{-\beta}(p\mathbb{Z}_{p})\subset H$.
Since also $x_{\alpha}(\mathbb{Z}_{p})\subset H$, $x_{-\gamma}(p\mathbb{Z}_{p})\subset H$
by Corollary~\ref{cor:SpanAlpha+Beta}.\end{proof}
\begin{rem}
From the Hasse diagrams in $\cite{Rg}$,
it seems that in the previous proof, we may always require $\alpha$
to be a simple root. 
\end{rem}
\begin{proof}
(Of theorem~\ref{thm:Main} for $\mathbf{G}$ simple, simply connected,
not of type $G_{2}$) By lemma~\ref{lem:SpanT}, \ref{lem:SpanI+},
\ref{lem:SpanI-} and the Iwahori decomposition of section~\ref{sub:IwaDec},
$I(1)$ is generated by 
\[
\left\{ x_{\alpha}(\mathbb{Z}_{p}):\alpha\in\Delta\right\} \cup\left\{ x_{-\alpha_{max}}(p\mathbb{Z}_{p})\right\} 
\]
thus topologically generated by 
\[
\left\{ x_{\alpha}(1):\alpha\in\Delta\right\} \cup\left\{ x_{-\alpha_{max}}(p)\right\} .
\]
None of these topological generators can be removed: the first ones
are contained in $I^{+}\subsetneq I(1)$, and all of them are needed
to span the image of 
\[
I(1)\twoheadrightarrow\mathbf{U}(\mathbb{F}_{p})\twoheadrightarrow\overline{\mathbf{U}}_{1}(\mathbb{F}_{p})\simeq\prod_{\alpha\in\Delta}\mathbf{U}_{\alpha}(\mathbb{F}_{p}),
\]
a surjective morphism that kills $x_{-\alpha_{max}}(p)$. 
\end{proof}

\subsection{~}\label{sub:caseofG2}

Let now $\mathbf{G}$ be simple of type $G_{2}$, thus $\Delta=\{\alpha,\beta\}$
with $\left\Vert \alpha\right\Vert <\left\Vert \beta\right\Vert $
and 
\[
R_{+}=\{\alpha,\beta,\beta+\alpha,\beta+2\alpha,\beta+3\alpha,2\beta+3\alpha\}.
\]
The whole root system looks like this:
 	
 \begin{center}
 	\begin{tikzpicture} [scale=.25]
 	\draw (0,0) -- (0:5) node [right] {$\alpha$} -- (7,4) node [right] { $3\alpha+\beta$}  -- cycle;
 	\draw (0,0) -- (0:-5) node [left] {$-\alpha$} -- (-7,4) node [below left] { $\beta$}  -- cycle;
 	\draw (0,0) -- (0:-5) -- (-7,4) -- cycle;
 	\draw (-5,0) --(-7,-4) node [above left] {$-3\alpha-\beta$}  --  (0,0) -- cycle;
 	\draw (5,0) --(7,-4) node [above right] {$-\beta$}  --  (0,0) -- cycle;
 	\draw (7,4) --(2.5,4) node [above right] {$2\alpha+\beta$}  --  (0,0) -- cycle;
 	\draw (-7,4) --(-2.5,4) node [above left] { $\alpha+\beta$}  --  (0,0) -- cycle;
 	\draw (-7,-4) --(-2.5,-4) node [below left]  {$-2\alpha-\beta$}  --  (0,0) -- cycle;
 	\draw (7,-4) --(2.5,-4) node [below right] { $-\alpha-\beta$}  --  (0,0) -- cycle;
 	\draw (2.5,-4) -- (0,-8) node [below] { $-3\alpha-2\beta$}  --  (0,0) -- cycle;
 	\draw (-2.5,4) -- (0,8) node [above] { $3\alpha+2\beta$}  --  (0,0) -- cycle;
 	\draw (0,8)  --  (2.5,4)  --  (0,0) -- cycle; 
 	\draw (0,-8)  --  (-2.5,-4)  --  (0,0) -- cycle;
 	
 	\end{tikzpicture}
 \end{center}
\begin{lem}
\label{lem:SpanI-G2}The group generated by $I^{+}$ and $x_{-2\beta-3\alpha}(p\mathbb{Z}_{p})$
contains $I^{-}$.\end{lem}
\begin{proof}
Let $H\subset I(1)$ be the group generated by $I^{+}$ and $x_{-2\beta-3\alpha}(p\mathbb{Z}_{p})$.
Then, for every $u,v\in\mathbb{Z}_{p}$, $H$ contains {\footnotesize{
\begin{eqnarray*}
[x_{-2\beta-3\alpha}(pv):x_{\beta}(u)] & = & x_{-\beta-3\alpha}(\pm puv)\\
{}[x_{-2\beta-3\alpha}(pv):x_{\beta+3\alpha}(u)] & = & x_{-\beta}(\pm puv)\\
{}[x_{-2\beta-3\alpha}(pv):x_{\beta+2\alpha}(u)] & = & x_{-\beta-\alpha}(\pm puv)\cdot x_{\alpha}(\pm pu^{2}v)\cdot x_{\beta+3\alpha}(\pm pu^{3}v)\cdot x_{-\beta}(\pm p^{2}u^{3}v^{2})
\end{eqnarray*}
}}It thus contains $x_{-\beta-3\alpha}(p\mathbb{Z}_{p})$, $x_{-\beta}(p\mathbb{Z}_{p})$
and $x_{-\beta-\alpha}(p\mathbb{Z}_{p})$, along with {\footnotesize{
\begin{eqnarray*}
[x_{-\beta-3\alpha}(pv):x_{\alpha}(u)] & = & x_{-\beta-2\alpha}(\pm puv)\cdot x_{-\beta-\alpha}(\pm pu^{2}v)\cdot x_{-\beta}(\pm pu^{3}v)\cdot x_{-2\beta-3\alpha}(\pm p^{2}u^{3}v^{2})\\
{}[x_{-\beta-3\alpha}(pv):x_{\beta+2\alpha}(u)] & = & x_{-\alpha}(\pm puv)\cdot x_{\beta+\alpha}(\pm pu^{2}v)\cdot x_{2\beta+3\alpha}(\pm pu^{3}v)\cdot x_{\beta}(\pm p^{2}u^{3}v^{2})
\end{eqnarray*}
}}It therefore also contains $x_{-\beta-2\alpha}(p\mathbb{Z}_{p})$
and $x_{-\alpha}(p\mathbb{Z}_{p})$.
\end{proof}
The filtration $(\mathbf{U}_{i})_{i\geq1}$ of $\mathbf{U}$ in section~\ref{sub:IwaDec}
induces a filtration 
\[
I^{+}=I_{1}^{+}\supset\cdots\supset I_{5}^{+}\supset I_{6}^{+}=1
\]
of $I^{+}=\mathbf{U}(\mathbb{Z}_{p})$ by normal subgroups $I_{i}^{+}=\mathbf{U}_{i}(\mathbb{Z}_{p})$
whose graded pieces 
\[
\overline{I}_{i}^{+}=\overline{\mathbf{U}}_{i}(\mathbb{Z}_{p})=I_{i}^{+}/I_{i+1}^{+}
\]
 are free $\mathbb{Z}_{p}$-modules, namely
\[
\begin{array}{c}
\overline{I}_{1}^{+}=\mathbb{Z}_{p}\cdot\overline{x}_{\alpha}\oplus\mathbb{Z}_{p}\cdot\overline{x}_{\beta},\qquad\overline{I}_{2}^{+}=\mathbb{Z}_{p}\cdot\overline{x}_{\alpha+\beta}\\
\overline{I}_{3}^{+}=\mathbb{Z}_{p}\cdot\overline{x}_{2\alpha+\beta},\qquad\overline{I}_{4}^{+}=\mathbb{Z}_{p}\cdot\overline{x}_{3\alpha+\beta},\qquad\overline{I}_{5}^{+}=\mathbb{Z}_{p}\cdot\overline{x}_{3\alpha+2\beta}
\end{array}
\]
 where $\overline{x}_{\gamma}$ is the image of $x_{\gamma}(1)$.
The commutator defines $\mathbb{Z}_{p}$-linear pairings 
\[
[-,-]_{i,j}:\overline{I}_{i}^{+}\times\overline{I}_{j}^{+}\rightarrow\overline{I}_{i+j}^{+}
\]
with $[y,x]_{j,i}=-[x,y]_{i,j}$, $[x,x]_{i,i}=0$ and, by Proposition~\ref{prop:ChevalleyFormulas},
\[
\begin{array}{c}
[\overline{x}_{\beta},\overline{x}_{\alpha}]=\pm\overline{x}_{\alpha+\beta},\quad[\overline{x}_{\alpha+\beta},\overline{x}_{\alpha}]=\pm2\overline{x}_{2\alpha+\beta},\quad[\overline{x}_{2\alpha+\beta},\overline{x}_{\alpha}]=\pm3\overline{x}_{3\alpha+\beta},\\{}
[\overline{x}_{\alpha+\beta},\overline{x}_{2\alpha+\beta}]=\pm x_{3\alpha+2\beta}\quad\mbox{and}\quad[\overline{x}_{\beta},\overline{x}_{3\alpha+\beta}]=\pm x_{2\alpha+2\beta}
\end{array}
\]
Let $H$ be the subgroup of $I^{+}$ generated by $x_{\alpha}(\mathbb{Z}_{p})$
and $x_{\beta}(\mathbb{Z}_{p})$ and denote by $H_{i}$ its image
in $I^{+}/I_{i+1}^{+}=G_{i}$. Then $H_{1}=G_{1}$, $H_{2}$ contains
$[\overline{x}_{\beta},\overline{x}_{\alpha}]=\pm\overline{x}_{\alpha+\beta}$
thus $H_{2}=G_{2}$, $H_{3}$ contains $[\overline{x}_{\alpha+\beta},\overline{x}_{\alpha}]=\pm2\overline{x}_{2\alpha+\beta}$
thus $H_{3}=G_{3}$ since $p\neq2$, $H_{4}$ contains $[\overline{x}_{2\alpha+\beta},\overline{x}_{\alpha}]=\pm3\overline{x}_{3\alpha+\beta}$
thus $H_{4}=G_{4}$ if $p\neq3$, in which case actually $H=H_{5}=G_{5}=I^{+}$
since $H$ always contains $[\overline{x}_{\alpha+\beta},\overline{x}_{2\alpha+\beta}]=\pm x_{3\alpha+2\beta}$. 

If $p=3$, let us also consider the exact sequence 
\[
0\rightarrow J_{4}\rightarrow G_{4}\rightarrow\overline{I}_{1}^{+}\rightarrow0
\]
The group $J_{4}=I_{2}^{+}/I_{5}^{+}$ is commutative, and in fact
again a free $\mathbb{Z}_{3}$-module: 
\[
J_{4}=(\mathbf{U}_{2}/\mathbf{U}_{5})(\mathbb{Z}_{p})=\mathbb{Z}_{3}\tilde{x}_{\alpha+\beta}\oplus\mathbb{Z}_{3}\tilde{x}_{2\alpha+\beta}\oplus\mathbb{Z}_{3}\overline{x}_{3\alpha+\beta}
\]
where $\tilde{x}_{\gamma}$ is the image of $x_{\gamma}(1)$. The
action by conjugation of $\overline{I}_{1}^{+}$ on $J_{4}$ is given
by 
\[
\overline{x}_{\alpha}\mapsto\left(\begin{array}{ccc}
1\\
\pm2 & 1\\
\pm3 & \pm3 & 1
\end{array}\right)\quad\overline{x}_{\beta}\mapsto\left(\begin{array}{ccc}
1\\
 & 1\\
 &  & 1
\end{array}\right)
\]
in the indicated basis of $J_{4}$. The $\mathbb{Z}_{3}$-submodule
$H'_{4}=H_{4}\cap J_{4}$ of $J_{4}$ satisfies 
\[
H'_{4}+\mathbb{Z}_{3}\overline{x}_{3\alpha+\beta}=J_{4}\quad\mbox{and}\quad3\overline{x}_{3\alpha+\beta}\in H'_{4}.
\]
Naming signs $\epsilon_{i}\in\{\pm1\}$ in formula $(1,3)$ of proposition~\ref{prop:ChevalleyFormulas},
we find that $H'_{4}$ contains 
\[
\epsilon_{1}uv\cdot\tilde{x}_{\alpha+\beta}+\epsilon_{2}u^{2}v\cdot\tilde{x}_{2\alpha+\beta}+\epsilon_{3}u^{3}v\cdot\overline{x}_{3\alpha+\beta}
\]
for every $u,v\in\mathbb{Z}_{3}$. Adding these for $v=1$ and $u=\pm1$,
we obtain 
\[
\tilde{x}_{2\alpha+\beta}\in H'_{4}.
\]
It follows that $H'_{4}$ actually contains the following $\mathbb{Z}_{3}$-submodule
of $J_{4}$: 
\[
J'_{4}=\left\{ a\cdot\tilde{x}_{\alpha+\beta}+b\cdot\tilde{x}_{2\alpha+\beta}+c\cdot\overline{x}_{3\alpha+\beta}:a,b,c\in\mathbb{Z}_{3},\,\epsilon_{1}a\equiv\epsilon_{3}c\bmod3\right\} .
\]
Now observe that $J'_{4}$ is a normal subgroup of $G_{4}$, and the
induced exact sequence
\[
0\rightarrow J_{4}/J'{}_{4}\rightarrow G_{4}/J'_{4}\rightarrow\overline{I}_{1}^{+}\rightarrow0
\]
is an \emph{abelian} extension of $\overline{I}_{1}^{+}\simeq\mathbb{Z}_{3}^{2}$
by $J_{4}/J'_{4}\simeq\mathbb{F}_{3}$. Since $H_{4}/J'_{4}$ is topologically
generated by two elements and surjects onto $\overline{I}_{1}^{+}$,
it actually defines a splitting: 
\[
G_{4}/J'_{4}=H_{4}/J'_{4}\oplus J_{4}/J'_{4}.
\]
Thus $H'_{4}=J'_{4}$, $H_{4}$ is a normal subgroup of $G_{4}$,
$H$ is a normal subgroup of $I^{+}$ and 
\[
I^{+}/H\simeq G_{4}/H_{4}\simeq J_{4}/J'_{4}\simeq\mathbb{F}_{3}
\]
is generated by the class of $x_{\alpha+\beta}(1)$ or $x_{3\alpha+\beta}(1)$.
We have shown:
\begin{lem}
\label{lem:SpanI+G2}The group $I^{+}$ is spanned by $x_{\alpha}(\mathbb{Z}_{p})$
and $x_{\beta}(\mathbb{Z}_{p})$ plus $x_{\alpha+\beta}(1)$ if $p=3$. \end{lem}
\begin{proof}
(Of theorem~\ref{thm:Main} for $\mathbf{G}$ simple of type $G_{2}$)
By lemma~\ref{lem:SpanT}, \ref{lem:SpanI-G2}, \ref{lem:SpanI+G2}
and the Iwahori decomposition of section~\ref{sub:IwaDec}, the pro-$p$-Iwahori
$I(1)$ is generated by $x_{\alpha}(\mathbb{Z}_{p})$, $x_{\beta}(\mathbb{Z}_{p})$,
$x_{-2\beta-3\alpha}(p\mathbb{Z}_{p})$, along with $x_{\alpha+\beta}(1)$
if $p=3$. It is therefore topologically generated by $x_{\alpha}(1)$,
$x_{\beta}(1)$, $x_{-2\beta-3\alpha}(p)$, along with $x_{\alpha+\beta}(1)$
if $p=3$. The surjective reduction morphism $I(1)\twoheadrightarrow\mathbf{U}(\mathbb{F}_{p})\twoheadrightarrow\overline{\mathbf{U}}_{1}(\mathbb{F}_{p})$
shows that the first two generators can not be removed. The third
one also can not, since all the others belong to the closed subgroup
$I_{+}\subsetneq I(1)$. Finally, suppose that $p=3$ and consider
the extension 
\[
1\rightarrow\mathbf{U}_{2}/\mathbf{U}_{5}\rightarrow\mathbf{U}/\mathbf{U}_{5}\rightarrow\mathbf{U}/\mathbf{U}_{1}\rightarrow1
\]
With notations as above, the reduction of 
\[
J'_{4}\subset J_{4}=\mathbf{U}_{2}(\mathbb{Z}_{3})/\mathbf{U}_{5}(\mathbb{Z}_{3})=(\mathbf{U}_{2}/\mathbf{U}_{5})(\mathbb{Z}_{3})
\]
is a normal subgroup $Y$ of $X=(\mathbf{U}/\mathbf{U}_{5})(\mathbb{F}_{3})$
with quotient $X/Y\simeq\mathbb{F}_{3}^{3}$. The surjective reduction
morphism 
\[
I(1)\twoheadrightarrow\mathbf{U}(\mathbb{F}_{3})\twoheadrightarrow\mathbf{U}(\mathbb{F}_{3})/\mathbf{U}_{5}(\mathbb{F}_{3})=X\twoheadrightarrow X/Y
\]
then kills $x_{-2\beta-3\alpha}(p)$. The fourth topological generator
$x_{\alpha+\beta}(1)$ of $I(1)$ thus also can not be removed, since
the first two certainly do not span $X/Y\simeq\mathbb{F}_{3}^{3}$.
\end{proof}

\subsection{~}\label{sub:Isogenies}

We now return to an arbitrary split reductive group $\mathbf{G}$
over $\mathbb{Z}_{p}$. Let 
\[
\mathbf{G}^{sc}\twoheadrightarrow\mathbf{G}^{der}\hookrightarrow\mathbf{G}\twoheadrightarrow\mathbf{G}^{ad}
\]
be the simply connected cover $\mathbf{G}^{sc}$ of the derived group
$\mathbf{G}^{der}$ of $\mathbf{G}$, and the adjoint group $\pi:\mathbf{G}\twoheadrightarrow\mathbf{G}^{ad}$
of $\mathbf{G}$. Then 
\[
\left(\mathbf{T}^{ad},M^{ad},R^{ad},\Delta^{ad},\left(X_{\alpha}^{ad}\right)_{\alpha\in\Delta^{ad}}\right)=\left(\pi(\mathbf{T}),\mathbb{Z}R,R,\Delta,\left(\pi(X_{\alpha})\right)_{\alpha\in\Delta}\right)
\]
is a pinning of $\mathbf{G}^{ad}$ and this construction yields a
bijection between pinnings of $\mathbf{G}$ and pinnings of $\mathbf{G}^{ad}$.
Applying this to $\mathbf{G}^{sc}$ or $\mathbf{G}^{der}$, we obtain
pinnings 
\[
\left(\mathbf{T}^{sc},M^{sc},R^{sc},\Delta^{sc},\left(X_{\alpha}^{sc}\right)_{\alpha\in\Delta^{sc}}\right)\quad\mbox{and}\quad\left(\mathbf{T}^{der},M^{der},R^{der},\Delta^{der},\left(X_{\alpha}^{der}\right)_{\alpha\in\Delta^{sc}}\right)
\]
for $\mathbf{G}^{sc}$ and $\mathbf{G}^{der}$: all of the above constructions
then apply to $\mathbf{G}^{ad}$, $\mathbf{G}^{sc}$ or $\mathbf{G}^{der}$,
and we will denote with a subscript $ad$, $sc$ or $der$ for the
corresponding objects. For instance, we have a sequence of Iwahori
(resp. pro-$p$-Iwahori) subgroups 
\[
I^{sc}\rightarrow I^{der}\hookrightarrow I\rightarrow I^{ad}\quad\mbox{and}\quad I^{sc}(1)\rightarrow I^{der}(1)\hookrightarrow I(1)\rightarrow I^{ad}(1).
\]

\subsection{~}

The action of $\mathbf{G}$ on itself by conjugation factors through
a morphism 
\[
\Ad:\mathbf{G}^{ad}\rightarrow\Aut(\mathbf{G}).
\]
For $b\in\mathbf{B}^{ad}(\mathbb{F}_{p})$, $\Ad(b)(\mathbf{B}_{\mathbb{F}_{p}})=\mathbf{B}_{\mathbb{F}_{p}}$
and $\Ad(b)(\mathbf{U}_{\mathbb{F}_{p}})=\mathbf{U}_{\mathbb{F}_{p}}$.
We thus obtain an action of the Iwahori subgroup $I^{ad}$ of $G^{ad}=\mathbf{G}^{ad}(\mathbb{Q}_{p})$
on $I$ or $I(1)$. Similar consideration of course apply to $\mathbf{G}^{sc}$
and $\mathbf{G}^{der}$, and the sequence 
\[
I^{sc}(1)\rightarrow I^{der}(1)\hookrightarrow I(1)\rightarrow I^{ad}(1)
\]
is equivariant for these actions of $I^{ad}=I^{ad}(1)\rtimes T_{tors}^{ad}$.

\subsection{~}

Let $J$ be the image of $I^{sc}(1)\rightarrow I(1)$, so that $J$
is a normal subgroup of $I$. From the compatible Iwahori decompositions
for $I(1)$ and $I^{sc}(1)$ in section~\ref{sub:IwaDec}, we see
that $T(1)\hookrightarrow I(1)$ induces a $T^{ad}$-equivariant isomorphism
\[
T(1)/T(1)\cap J\rightarrow I(1)/J.
\]
Since the inverse image of $\mathbf{T}(\mathbb{Z}_{p})$ in $\mathbf{G}^{sc}(\mathbb{Z}_{p})$
equals $\mathbf{T}^{sc}(\mathbb{Z}_{p})$ and since also 
\[
T^{sc}(1)=\mathbf{T}^{sc}(\mathbb{Z}_{p})\cap I^{sc}(1),
\]
we see that $T(1)\cap J$ is the image of $T^{sc}(1)\rightarrow T(1)$.
Also, the kernel of $I^{sc}(1)\rightarrow I(1)$ equals
$Z\cap I^{sc}(1)$ where 
\[
Z=\ker(\mathbf{G}^{sc}\rightarrow\mathbf{G})(\mathbb{Z}_{p})=\ker(\mathbf{T}^{sc}\rightarrow\mathbf{T})(\mathbb{Z}_{p}).
\]
Therefore $Z\cap I^{sc}(1)$ is the kernel of $T^{sc}(1)\rightarrow T(1)$,
which is trivial since $Z$ is finite and $T^{sc}(1)\simeq\Hom(M^{sc},1+p\mathbb{Z}_{p})$
has no torsion. We thus obtain exact sequences
\[
\begin{array}{ccccccccc}
1 & \rightarrow & T^{sc}(1) & \rightarrow & T(1) & \rightarrow & Q & \rightarrow & 0\\
 &  & \cap &  & \cap &  & \parallel\\
1 & \rightarrow & I^{sc}(1) & \rightarrow & I(1) & \rightarrow & Q & \rightarrow & 0
\end{array}
\]
where the cokernel $Q$ is the finitely generated $\mathbb{Z}_{p}$-module
\[
Q=\left(M^{\vee}/\mathbb{Z}R^{\vee}\right)\otimes\left(1+p\mathbb{Z}_{p}\right).
\]

\begin{rem}
If $\mathbf{G}$ is simple, then $M^{\vee}/\mathbb{Z}R^{\vee}$ is
a finite group of order $c$, with $c\mid\ell+1$ if $\mathbf{G}$
is of type $A_{\ell}$, $c\mid3$ if $\mathbf{G}$ is of type $E_{6}$
and $c\mid4$ in all other cases. Thus $Q=0$ and $I^{sc}(1)=I(1)$
unless $\mathbf{G}$ is of type $A_{\ell}$ with $p\mid c\mid\ell+1$
or $p=3$ and $\mathbf{G}$ is adjoint of type $E_{6}$. In these
exceptional cases, $M^{\vee}/\mathbb{Z}R^{\vee}$ is cyclic, thus
$Q\simeq\mathbb{F}_{p}$. 
\end{rem}

\subsection{~}

It follows that $I(1)$ is generated by $I^{sc}(1)$ and $s(1+p\mathbb{Z}_{p})$
for $s\in\mathcal{S}$, thus topologically generated by $I^{sc}(1)$
and $s(1+p)$ for $s\in\mathcal{S}$. In view of the results already
established in the simply connected case, this shows that the elements
listed in $(1-4)$ of Theorem~\ref{thm:Main} indeed form a set of
topological generators for $I(1)$. 

None of the semi-simple elements in $(1)$ can be removed: they are
all needed to generate the above abelian quotient $Q$ of $I(1)$
which indeed kills the unipotent generators in $(2-4)$. Likewise,
none of the unipotent elements in $(2)$ can be removed: they are
all needed to generate the abelian quotient 
\[
I(1)\twoheadrightarrow\mathbf{U}(\mathbb{F}_{p})\twoheadrightarrow\overline{\mathbf{U}}_{1}(\mathbb{F}_{p})\simeq\prod_{\alpha\in\Delta}\mathbf{U}_{\alpha}(\mathbb{F}_{p})
\]
which kills the other generators in $(1)$, $(3)$ and $(4)$. One
checks easily using the Iwahori decomposition of $I(1)$ and the product
decomposition $\mathbf{U}^{-}=\prod_{c\in\mathcal{C}}\mathbf{U}_{c}^{-}$
that none of the unipotent elements in $(3)$ can be removed. Finally
if $p=3$ and $d\in\mathcal{D}$, the central isogeny $\mathbf{G}^{sc}\rightarrow\mathbf{G}^{ad}$
induces an isomorphism $\mathbf{G}_{d}^{sc}\rightarrow\mathbf{G}_{d}^{ad}$
between the simple (simply connected \emph{and }adjoint) components
corresponding to $d$, thus also an isomorphism between the corresponding
pro-$p$-Iwahori's $I_{d}^{sc}(1)\rightarrow I_{d}^{ad}(1)$. In particular,
the projection $I(1)\rightarrow I^{ad}(1)\twoheadrightarrow I_{d}^{ad}(1)$
is surjective. Composing it with the projection $I_{d}^{ad}(1)\twoheadrightarrow\mathbb{F}_{3}^{3}$
constructed in section~\ref{sub:caseofG2}, we obtain an abelian
quotient $I(1)\twoheadrightarrow\mathbb{F}_{3}^{3}$ that kills all
of our generators except $x_{\alpha}(1)$, $x_{\beta}(1)$ and $x_{\alpha+\beta}(1)$
where $\Delta_{d}=\{\alpha,\beta\}$. In particular, the generator
$x_{\alpha+\beta}(1)$ from $(4)$ is also necessary. This finishes the proof of Theorem $\ref{thm:Main}$.

\subsection{~}

The action of $I^{ad}=I^{ad}(1)\rtimes T_{tors}^{ad}$ on $I(1)$
induces an $\mathbb{F}_{p}$-linear action of 
\[
T_{tors}^{ad}=\Hom\left(M^{ad},\mu_{p-1}\right)=\Hom\left(\mathbb{Z}R,\mathbb{F}_{p}^{\times}\right)
\]
on the Frattini quotient $\tilde{I}(1)$ of $I(1)$. Our minimal set
of topological generators of $I(1)$ reduces to an eigenbasis of $\tilde{I}(1)$,
i.e. an $\mathbb{F}_{p}$-basis of $\tilde{I}(1)$ made of eigenvectors
for the action of $T_{tors}^{ad}$. We denote by $\mathbb{F}_{p}(\alpha)$
the $1$-dimensional representation of $T_{tors}^{ad}$ on $\mathbb{F}_{p}$
defined by $\alpha\in\mathbb{Z}R$. We thus obtain:
\begin{cor}\label{cor:submain}
The $\mathbb{F}_{p}[T_{tors}^{ad}]$-module $\tilde{I}(1)$ is isomorphic
to
\[
\mathbb{F}_{p}^{\sharp \mathcal S}\oplus\Big(\oplus_{\alpha\in\Delta}\mathbb{F}_{p}(\alpha)\Big)\oplus\Big(\oplus_{c\in\mathcal{C}}\mathbb{F}_{p}(-\alpha_{c,max})\Big)\color{magenta}\Big(\oplus\Big(\oplus_{d\in\mathcal{D}}\mathbb{F}_{p}(\delta_{c})\Big)\mbox{\, if }p=3\Big).
\]

\end{cor}
\noindent Here $\sharp \mathcal{S}$ denotes the cardinality of the set $\mathcal{S}$. The map $\alpha\mapsto\mathbb{F}_{p}(\alpha)$ yields a
bijection between $\mathbb{Z}R/(p-1)\mathbb{Z}R$ and the isomorphism
classes of simple $\mathbb{F}_{p}[T_{tors}^{ad}]$-modules. In particular some of the simple modules in the
previous corollary may happen to be isomorphic. For instance if $\mathbf{G}$
is simple of type $B_{\ell}$ and $p=3$, then $-\alpha_{max}\equiv\alpha\bmod2$
where $\alpha\in\Delta$ is a long simple root. An inspection
of the tables in \cite{BoLie46}  yields the following:
\begin{cor}\label{cor:good}
If $\mathbf{G}$ is simple, the $\mathbb{F}_{p}[T_{tors}^{ad}]$-module
$\tilde{I}(1)$ is multiplicity free unless $p=3$ and $\mathbf{G}$
is of type $A_{1}$, $B_{\ell}$ or $C_{\ell}$ ($\ell\geq2$), $F_{4}$
or $G_{2}$. 
\end{cor}

 In the next section we use this result to construct Galois representations landing in $I^{ad}$ with image containing $I^{ad}(1)$.
\section{The Construction of Galois Representations}

Let $\mathbf{G}$ be a split simple adjoint group over $\mathbb{Z}_{p}$
and let $I(1)$ and $I=I(1)\rtimes T_{tors}$ be the corresponding
Iwahori groups, as defined in the previous section. We want here to construct Galois representations of a certain type with values in $I$ with image containing $I(1)$. After a short review of $p$-rational fields in section $3.1$, we establish a criterion for the existence of our representations in sections $3.2$ and $3.3$ and finally give some examples in section $3.4$.

\subsection{~} 
Let $K$ be a number field, $r_2(K)$ the number of complex primes of $K$, $\Sigma_p$  the set of primes of $K$ lying above $p$, $M$ the compositum of all finite $p$-extensions of $K$ which are unramified outside $\Sigma_p$, $M^{ab}$ the maximal abelian extension of
$K$ contained in $M$,  and $L$  the compositum of all cyclic extensions of $K$ of degree $p$ which are
contained in $M$ or $M^{ab}$.  If we let $\Gamma$ denote Gal$(M/K)$, then $\Gamma$ is a pro-$p$ group, $\Gamma^{ab}\cong$ Gal$(M^{ab}/K)$  is the maximal abelian quotient of $\Gamma$, and $\tilde{\Gamma}\cong \Gamma^{ab}/p\Gamma^{ab} \cong  \text{Gal}(L/K)$ is the Frattini quotient of $\Gamma$.\\

\textbf{Definition} \textit{A number field} $K$ \textit{is} $p$\textit{-rational if the following
equivalent conditions are satisfied:}
\begin{flushleft}
$(1)$ \textit{rank}$_{\mathbb{Z}_p}(\Gamma^{ab})=r_2(K)+1$ \textit{and} $\Gamma^{ab}$  \textit{is torsion-free as a} $\mathbb{Z}_p$\textit{-module},\\
$(2)$ $\Gamma$ \textit{is a free pro-}$p$ \textit{group with} $r_2(K)+1$ \textit{generators,}\\
$(3)$ $\Gamma$ \textit{is a free pro-}$p$ \textit{group}.
\end{flushleft}
The equivalence of $(1),(2)$ and $(3)$ follows from $\cite{Moh}$, see also proposition $3.1$ and the discussion before remark 3.2 of $\cite{Green}$. There is a considerable literature concerning $p$-rational fields, including $\cite{Mov},\cite{JaNg}$.\\

\textbf{Examples:} 

$(1)$ Suppose that $K$ is a quadratic field and that either $p \geq 5$ or $p=3$ and is unramified in $K/\mathbb{Q}$. If $K$ is real, then $K$
is $p$-rational if and only if $p$ does not divide the class number of $K$ and the fundamental unit of $K$ is not a $p$-th power in the completions $K_v$ of $K$ at the places $v$ above $p$. On the other hand, if $K$ is complex and $p$ does not divide the class number of $K$, then $K$ is a $p$-rational field (cf. proposition $4.1$ of $\cite{Green}$). However, there are $p$-rational complex $K$'s for which $p$ divides the class number (cf. chapter $2$, section $1$, p. $25$ of $\cite{Thesis}$). For similar results, see also $\cite{Fuji}$ and  $\cite{Min}$ if $K$ is complex.

$(2)$ Let $K=\mathbb{Q}(\mu_{p})$. If $p$ is a regular prime, then $K$ is a $p$-rational field (cf. \cite{Shav}, see also \cite{Green}, proposition $4.9$ for a shorter proof).\\

\subsection{~}
Suppose that $K$ is Galois over $\mathbb{Q}$ and $p$-rational with
$p\nmid[K:\mathbb{Q}]$. 

Since $K$ is Galois over $\mathbb{Q}$, so is $M$ and we have an
exact sequence
\begin{equation}\label{eq:exact}
1 \rightarrow \Gamma \rightarrow \Pi \rightarrow \Omega \rightarrow 1
\end{equation}
where $\Omega=$ Gal$(K/\mathbb{Q})$ and $\Pi=$ Gal$(M/\mathbb{Q})$. Conjugation in
$\Pi$ then induces an action of $\Omega$ on the Frattini quotient
$\tilde{\Gamma}=$ Gal$(L/K)$ of $\Gamma$. Any continuous morphism $\rho:\Pi\rightarrow I$
maps $\Gamma$ to $I(1)$ and induces a morphism $\overline{\rho}: \Omega \rightarrow I/I(1)=T_{tors}$
and a $\overline{\rho}$-equivariant morphism $\tilde{\rho}:\tilde{\Gamma}\rightarrow\tilde{I}(1)$.
If $\rho(\Gamma)=I(1)$, then $\tilde{\rho}$ is also surjective.
Suppose conversely that we are given the finite data
\[
\overline{\rho}:\Omega\rightarrow T_{tors}\quad\mbox{and}\quad\tilde{\rho}:\tilde{\Gamma}\twoheadrightarrow\tilde{I}(1).
\]
Then as 
$\Omega$ has order prime to $p$, the Schur-Zassenhaus theorem ($\cite{Wilson}$, proposition $2.3.3$)
implies that the exact sequence $\ref{eq:exact}$ splits. The choice of a splitting $\Pi\simeq\Gamma\rtimes\Omega$
yields a non-canonical action of $\Omega$ on $\Gamma$ which lifts
the canonical action of $\Omega$ on the Frattini quotient $\tilde{\Gamma}$.
By $\cite{Green}$, proposition~$2.3$, $\tilde{\rho}$ lifts to a continuous $\Omega$-equivariant surjective
morphism $\rho':\Gamma\twoheadrightarrow I(1)$, which plainly gives
a continuous morphism
\begin{equation*}
\rho=(\rho',\overline{\rho}):\Pi\simeq\Gamma\rtimes\Omega\rightarrow I=I(1)\rtimes T_{tors}
\end{equation*}
inducing $\overline{\rho}:\Omega\rightarrow T_{tors}$ and $\tilde{\rho}:\tilde{\Gamma}\twoheadrightarrow\tilde{I}(1)$.
Thus:

\begin{prop}\label{prop:after}
Under the above assumptions on $K$, there is a continuous morphism $\rho:\Pi\rightarrow I$ such that $\rho(\Gamma)=I(1)$
if and only if there is a morphism $\overline{\rho}:\Omega\rightarrow T_{tors}$
such that the induced $\mathbb{F}_{p}[\Omega]$-module $\overline{\rho}^{\ast}\tilde{I}(1)$
is a quotient of $\tilde{\Gamma}$. 
\end{prop}
 \noindent The Frattini quotient $\tilde{I}(1)$ is an $\mathbb{F}_p[T_{tors}]$-module and by the map $\overline{\rho}$, we can consider $\tilde{I}(1)$ as an $\mathbb{F}_p[\Omega]$-module which we denote by $\overline{\rho}^{\ast}\tilde{I}(1)$.
\subsection{~}

Suppose now that 
\begin{flushleft}
$ {\large \textbf{A($K$):}}$ \ $K$ is a totally complex abelian (thus CM) Galois extension
of $\mathbb{Q}$ which is $p$-rational of 

\hspace{1.5cm} \  degree $[K:\mathbb{Q}]\mid p-1$. 

\end{flushleft}

Let $\hat{\Omega}$ be the group of characters of $\Omega$ with values
in $\mathbb{F}_{p}^{\times}$, $\hat{\Omega}_{odd}\subset\hat{\Omega}$
the subset of odd characters (those taking the value $-1$ on complex
conjugation), and $\chi_{0}\in\hat{\Omega}$ the trivial character.
Then by $\cite{Green}$ proposition $3.3$,
\[
\tilde{\Gamma}=\oplus_{\chi \in \hat{\Omega}_{odd} \cup \{\chi_0\}}\mathbb{F}_p(\chi)
\]
as an $\mathbb{F}_{p}[\Omega]$-module. In particular, $\tilde{\Gamma}$
is multiplicity free. Suppose therefore also that the $\mathbb{F}_{p}[T_{tors}]$-module
$\tilde{I}(1)$ is multiplicity free, i.e. by corollary $\ref{cor:good}$,

\begin{flushleft}
$ {\large \textbf{B($G$):}}$ \ \ If $p=3$, then $\mathbf{G}$ is not of type $A_{1}$,
$B_{\ell}$ or $C_{\ell}$ ($\ell\geq2$), $F_{4}$ or $G_{2}$. 
\end{flushleft}

\noindent For $\mathcal{S}$ as in section $2.4$, we define
\[
    \hat{\Omega}^{\mathcal{S}}_{odd}= 
\begin{cases}
    \hat{\Omega}_{odd} \cup \chi_0,& \text{if } {\mathcal{S}}=\emptyset\\
    \hat{\Omega}_{odd},              & \text{if }  {\mathcal{S}}\neq \emptyset.
\end{cases}
\]

 Note that $\mathcal{S}=\emptyset$
unless $\mathbf{G}$ if of type $A_{\ell}$ with $p\mid\ell+1$ or
$\mathbf{G}$ is of type $E_{6}$ with $p=3$, in which both cases
$\mathcal{S}$ is a singleton. We thus obtain: 
\begin{cor}\label{cor:main}
Under the assumptions $ {\textbf{A($K$)}}$ on $K$ and
$ { \textbf{B($G$)}}$ on $\mathbf{G}$, there is a morphism $\rho:\Pi\rightarrow I$
  such that $\rho(\Gamma)=I(1)$ if and only if there is morphism $\overline{\rho}:\Omega\rightarrow T_{tors}$
  such that the characters $\alpha\circ\overline{\rho}:\Omega\rightarrow\mathbb{F}_{p}^{\times}$
  for $\alpha\in \Delta \cup \{-\alpha_{max}\}$ are all distinct and belong to $\hat{\Omega}^{\mathcal{S}}_{odd}$.
\end{cor}

\subsection{Some examples.}
Write $\Delta=\{\alpha_1,...,\alpha_{\ell}\}$ and  $\alpha_{max}=n_1\alpha_1+\cdots +n_{\ell}\alpha_{\ell}$ using the conventions of the tables in
$\cite{BoLie46}$. In this part we suppose that $p$ is a regular (odd) prime
and take $K=\mathbb{Q}(\mu_{p})$, so that $K$ is $p$-rational and
$\Omega=\mathbb{Z}/(p-1)\mathbb{Z}$.
\begin{lem}\label{eq:findcharacters}
Suppose $\mathbf{G}$ is of type $A_{\ell},B_{\ell},C_{\ell}$ or $D_{\ell}$ and  $p \geq 2l+3$ (resp. $p \geq 2l+5$) if $p \equiv 1\mod 4$ (resp. $p \equiv 3\mod 4$). Then we can find distinct characters $\phi_1,...,\phi_{{\ell}+1} \in \hat{\Omega}_{odd} \cup \chi_0$ such that $\phi_1^{n_1}\phi_2^{n_2}\cdots \phi_{{\ell}}^{n_{\ell}}\phi_{{\ell}+1}=\chi_0$. Furthermore, if $\mathbf{G}$ is of type $A_{\ell}$ and ${\ell}$ is odd, then one can even choose the characters $\phi_1,...,\phi_{{\ell}+1}$ to be inside $\hat{\Omega}_{odd}$.
\end{lem}
\begin{proof}
Since $\Omega$ is (canonically) isomorphic to $\mathbb{Z}/(p-1)\mathbb{Z}$,
$\sharp \hat{\Omega}_{odd}=\frac{p-1}{2}$ and there are exactly $[\frac{p-1}{4}]$
pairs of characters $\{\chi,\chi^{-1}\}$ with $\chi\neq\chi^{-1}$
in $\hat{\Omega}_{odd}$. The condition on $p$ is equivalent to ${\ell} \leq 2[\frac{p-1}{4}]-1$.

If $\mathbf{G}$ is of type $A_{\ell}$, then $\alpha_{max}=\alpha_{1}+\cdots+\alpha_{\ell}$.
If $\ell$ is even and $\frac{\ell}{2}\leq[\frac{p-1}{4}]$, then
we can pick $\frac{\ell}{2}$ distinct pairs of odd characters $\{\chi,\chi^{-1}\}$
as above for $\{\phi_{1},\cdots,\phi_{\ell}\}$ and set $\phi_{\ell+1}=\chi_{0}$.
If $\ell$ is odd and $\frac{\ell+1}{2}\leq[\frac{p-1}{4}]$, then
we can choose $\frac{\ell+1}{2}$ distinct such pairs for the whole
set $\{\phi_{1},\cdots,\phi_{\ell+1}\}$.

If $\mathbf{G}$ is of type $D_{\ell}$ (with $\ell\geq4$), then $\alpha_{max}=\alpha_1+2\alpha_2+...+2\alpha_{{\ell}-2}+\alpha_{{\ell}-1}+\alpha_{\ell}$.
 Now if $\ell$ is odd we can pick $\frac{\ell+1}{2}$ such pairs $\{\chi,\chi^{-1}\}$, one for  $\{\phi_{{\ell}-1},\phi_{\ell}\}$, another pair for  $\{\phi_1,\phi_{{\ell}+1}\}$ and $\frac{\ell-3}{2}$ such pairs for  $\{\phi_2,...,\phi_{{\ell}-2}\}$. If ${\ell}$ is even, we let $\phi_2$  be the trivial character, and we can choose $\frac{\ell}{2}$ such pairs of characters $\{\chi,\chi^{-1}\}$, one pair for  $\{\phi_1,\phi_{{\ell}-1}\}$, another pair for  $\{\phi_{\ell},\phi_{{\ell}+1}\}$ and $\frac{\ell-4}{2}$ such pairs for  $\{\phi_3,...,\phi_{{\ell}-2}\}$. So the inequality that we will need is $4 \leq {\ell} \leq 2[\frac{p-1}{4}]-1$.

If $\mathbf{G}$ is of type $B_{\ell}$ (with $\ell\geq2$), then $\alpha_{max}=\alpha_1+2\alpha_2+...+2\alpha_{\ell}$. If
 ${\ell}$ is odd then we pick $\frac{\ell+1}{2}$ pairs of characters $\{\chi,\chi^{-1}\}$; one pair for  $\{\phi_1,\phi_{{\ell}+1}\}$ and $\frac{\ell-1}{2}$ such pairs for  $\{\phi_2,...,\phi_{\ell}\}$. If ${\ell}$ is even then we need $\frac{\ell}{2}$ pairs of $\{\chi,\chi^{-1}\}$; one pair for  $\{\phi_1,\phi_{{\ell}+1}\}$ and $\frac{\ell-2}{2}$ such pairs for  $\{\phi_3,...,\phi_{\ell}\}$ and we let $\phi_2$ be the trivial character. So in this case we need $3 \leq {\ell} \leq 2[\frac{p-1}{4}]-1$.
 
 The remaining $C_{\ell}$ case is analogous.
\end{proof}
\begin{lem}\label{lem:find}
  Suppose $\mathbf{G}$ is of type $E_6,E_7,E_8,F_4$ or $G_2$ and $p \geq \sum_{i=1}^{\ell}(2i-1)n_i+2{\ell}$. Then we can find distinct characters $\phi_1,...,\phi_{\ell+1} \in \hat{\Omega}_{odd}$ such that $\phi_1^{n_1}\phi_2^{n_2}\cdots \phi_{{\ell}}^{n_{\ell}}\phi_{{\ell}+1}=\chi_0$.
 \end{lem}
\begin{proof}
The choice of a generator $\xi$ of $\mathbb{F}_{p}^{\times}$
yields an isomorphism $\mathbb{Z}/(p-1)\mathbb{Z}\simeq\hat{\Omega}$,
mapping $i$ to $\chi_{i}$ and $1+2\mathbb{Z}/(p-1)\mathbb{Z}$ to
$\hat{\Omega}_{odd}$. Set $\phi_{i}=\chi_{2i-1}\in\hat{\Omega}_{odd}$
for $i=1,\cdots,\ell$ and $\phi_{\ell+1}=\chi_{-r}$ where $r={\sum_{i=1}^{\ell}n_{i}\cdot(2i-1)}$.
The tables in $\cite{BoLie46}$ show that $h=\sum_{i=1}^{\ell}n_{i}$ is odd,
thus also $\phi_{\ell+1}\in\hat{\Omega}_{odd}$ and plainly $\phi_{1}^{n_{1}}\cdots\phi_{\ell}^{n_{\ell}}\phi_{\ell+1}=1$.
If $p\geq \sum_{i=1}^{\ell}(2i-1)n_i+2{\ell}$, the elements $\{2i-1,-\sum_{i=1}^{\ell}n_{i}\cdot(2i-1);i\in [1,{\ell}]\}$ are all distinct modulo $p-1$, which
proves the lemma.
\end{proof}
\begin{rem}\label{rem:exceptional}
For $\mathbf{G}$ of type $E_6,E_7,E_8,F_4$ or $G_2$, the tables in $\cite{BoLie46}$ show that the constant  $\sum_{i=1}^{\ell}(2i-1)n_i+2{\ell}$ of lemma $\ref{lem:find}$ is $79,127,247,53,13$ respectively.   
\end{rem}
\begin{cor}\label{cor:lastcor}
There is a constant $c$ depending only upon the
type of $\mathbf{G}$ such that if $p>c$ is a regular prime, then
for $K=\mathbb{Q}(\mu_{p})$, $M$, $\Pi$ and $\Gamma$ as above,
there is a continuous morphism $\rho:\Pi\rightarrow I$ with $\rho(\Gamma)=I(1)$. 
\end{cor}
In conclusion, we have determined a minimal set of topological generators of the pro-$p$ Iwahori subgroup of a split reductive groups over $\mathbb{Z}_p$ (theorem $\ref{thm:Main}$) and used it to study the structure of the Frattini quotient $\tilde{I}(1)$ as an $\mathbb{F}_p[T_{tors}^{ad}]$-module (corollary $\ref{cor:submain}$). Then we have used corollary $\ref{cor:submain}$ to determine when $\tilde{I}(1)$ is multiplicity free (see corollary $\ref{cor:good}$). Furthermore in proposition $\ref{prop:after}$ and corollary $\ref{cor:main}$, assuming $p$-rationality, we have shown that we can construct Galois representations if and only if we can find a suitable list of distinct characters in $\Omega$, the existence of which is discussed in section $3.4$ under the assumption $K=\mathbb{Q}(\mu_p)$, for any sufficiently large regular prime $p$ (see corollary $\ref{cor:lastcor}$).

\end{document}